\documentclass{amsart}
\usepackage{amscd,amsmath,xypic,amssymb,combelow,tikz-cd,etoolbox,calligra,mathrsfs,stmaryrd,enumitem}

\DeclareMathOperator{\sHom}{\mathscr{H}\text{\kern -3pt {\calligra\large om}}\,}
\DeclareMathOperator{\sRHom}{\mathscr{RH}\text{\kern -3pt {\calligra\large om}}\,}
\DeclareMathOperator{\sQuot}{\mathscr{Q}\text{\kern -3pt {\calligra\large uot}}\,}

\makeatletter
\patchcmd{\@settitle}{\uppercasenonmath\@title}{}{}{}
\makeatother

\xyoption{all}
\CompileMatrices

\newcommand{\nc}{\newcommand}

\makeatletter
\@addtoreset{equation}{section}
\makeatother

\newtheorem{theorem}[subsection]{Theorem}
\newtheorem{proposition}[subsection]{Proposition}

\newtheorem{definition}[subsection]{Definition}

\nc{\fa}{{\mathfrak{a}}}
\nc{\fb}{{\mathfrak{b}}}
\nc{\fd}{{\mathfrak{d}}}
\nc{\fg}{{\mathfrak{g}}}
\nc{\fh}{{\mathfrak{h}}}
\nc{\fj}{{\mathfrak{j}}}
\nc{\fn}{{\mathfrak{n}}}
\nc{\fm}{{\mathfrak{m}}}
\nc{\fu}{{\mathfrak{u}}}
\nc{\fp}{{\mathfrak{p}}}
\nc{\fr}{{\mathfrak{r}}}
\nc{\ft}{{\mathfrak{t}}}
\nc{\fsl}{{\mathfrak{sl}}}
\nc{\fgl}{{\mathfrak{gl}}}
\nc{\hsl}{{\widehat{\mathfrak{sl}}}}
\nc{\hgl}{{\widehat{\mathfrak{gl}}}}
\nc{\hg}{{\widehat{\mathfrak{g}}}}
\nc{\chg}{{\widehat{\mathfrak{g}}}{}^\vee}
\nc{\hn}{{\widehat{\mathfrak{n}}}}
\nc{\chn}{{\widehat{\mathfrak{n}}}{}^\vee}

\nc{\Mod}{{\textrm{Mod}}}

\nc{\wGL}{{\widehat{GL}^+}}

\nc{\BA}{{\mathbb{A}}}
\nc{\BC}{{\mathbb{C}}}
\nc{\BG}{{\mathbb{G}}}
\nc{\BM}{{\mathbb{M}}}
\nc{\BN}{{\mathbb{N}}}
\nc{\BF}{{\mathbb{F}}}
\nc{\BH}{{\mathbb{H}}}
\nc{\BP}{{\mathbb{P}}}
\nc{\BQ}{{\mathbb{Q}}}
\nc{\BR}{{\mathbb{R}}}
\nc{\BZ}{{\mathbb{Z}}}

\nc{\ff}{{\mathbb{F}}}
\nc{\kk}{{\mathbb{K}}}
\nc{\kko}{{\mathbb{K}}}

\nc{\coh}{{\text{Coh}}}

\nc{\CA}{{\mathcal{A}}}
\nc{\CC}{{\mathcal{C}}}
\nc{\CB}{{\mathcal{B}}}
\nc{\DD}{{\mathcal{D}}}
\nc{\CE}{{\mathcal{E}}}
\nc{\CF}{{\mathcal{F}}}
\nc{\tCF}{{\widetilde{\CF}}}
\nc{\tCM}{{\widetilde{\CM}}}
\nc{\tCT}{{\widetilde{\CT}}}
\nc{\oCF}{{\bar{\CF}}}
\nc{\CG}{{\mathcal{G}}}
\nc{\CL}{{\mathcal{L}}}
\nc{\CK}{{\mathcal{K}}}
\nc{\CI}{{\mathcal{I}}}
\nc{\CM}{{\mathcal{M}}}
\nc{\CH}{{\mathcal{H}}}
\nc{\CN}{{\mathcal{N}}}
\nc{\CO}{{\mathcal{O}}}
\nc{\CP}{{\mathcal{P}}}
\nc{\CR}{{\mathcal{R}}}
\nc{\CQ}{{\mathcal{Q}}}
\nc{\CS}{{\mathcal{S}}}
\nc{\CT}{{\mathcal{T}}}
\nc{\tCU}{{\widetilde{\CU}}}
\nc{\CU}{{\mathcal{U}}}
\nc{\CV}{{\mathcal{V}}}
\nc{\CW}{{\mathcal{W}}}
\nc{\CZ}{{\mathcal{Z}}}

\nc{\tpsi}{{\widetilde{\Psi}}}
\nc{\wpi}{{\widetilde{\pi}}}
\nc{\Ker}{{\text{Ker }}}
\nc{\Coker}{{\text{Coker }}}

\nc{\CX}{{\mathcal{X}}}
\nc{\tCX}{{\widetilde{\mathcal{X}}}}
\nc{\CY}{{\mathcal{Y}}}
\nc{\tCY}{{\widetilde{\mathcal{Y}}}}

\nc{\tN}{{\widetilde{\CN}}}
\nc{\pN}{{\BP\widetilde{\CN}}}

\nc{\tT}{{T}}

\nc{\fq}{{\mathfrak{q}}}

\nc{\fC}{{\mathfrak{C}}}
\nc{\fJ}{{\mathfrak{J}}}
\nc{\fG}{{\mathfrak{G}}}
\nc{\fL}{{\mathfrak{L}}}
\nc{\fZ}{{\mathfrak{Z}}}
\nc{\fU}{{\mathfrak{U}}}
\nc{\fV}{{\mathfrak{V}}}
\nc{\fS}{{\mathfrak{S}}}

\nc{\od}{{\overline{d}}}
\nc{\rg}{{\textrm{R}\Gamma}}
\nc{\erg}{{\emph{R}\Gamma}}
\nc{\id}{{\textrm{Id}}}
\nc{\rhom}{{\textrm{RHom}}}

\def\tCF{\widetilde{\CF}}

\def\taut{\textrm{univ}}
\def\etaut{\emph{univ}}

\def\fgl{\mathfrak{gl}}

\def\fb{\mathfrak{b}}
\def\fh{\mathfrak{h}}
\def\fg{\mathfrak{g}}
\def\fp{\mathfrak{p}}
\def\ft{\mathfrak{t}}

\def\Id{\text{Id}}

\def\Hilb{{\text{Hilb}}}
\def\eHilb{{\emph{Hilb}}}

\def\ch{{\text{ch}}}
\def\ech{{\emph{ch}}}

\begin{document}

\title[The Chow of $S^{[n]}$ and the universal subscheme]{\large{THE CHOW OF $S^{[n]}$ AND THE UNIVERSAL SUBSCHEME}}
\author[Andrei Negu\cb t]{Andrei Negu\cb t}

\address{MIT, Department of Mathematics, 77 Massachusetts Ave, Cambridge, MA 02139, USA}
\address{Simion Stoilow Institute of Mathematics, Bucharest, Romania}
\email{andrei.negut@gmail.com}

\renewcommand{\thefootnote}{\fnsymbol{footnote}} 
\renewcommand{\thefootnote}{\arabic{footnote}} 

\renewcommand{\thefootnote}{\fnsymbol{footnote}} 
\renewcommand{\thefootnote}{\arabic{footnote}} 

\maketitle

\begin{abstract}  We prove that any element in the Chow ring of the Hilbert scheme $\Hilb_n$ of $n$ points on a smooth surface $S$ is a universal class, i.e. the push-forward of a polynomial in the Chern classes of the universal subschemes on $\Hilb_n \times S^k$ for some $k \in \BN$, with coefficients pulled back from the Chow of $S^k$.

\end{abstract}

\section{Introduction}

\noindent  We study the Hilbert scheme of $n$ points $\Hilb_n = S^{[n]}$ on a smooth algebraic surface over $\BC$. As $\Hilb_n$ is smooth, we may consider the Chow groups $A^*(\Hilb_n)$, always with coefficients in $\BQ$ throughout the present paper. One of the big sources of elements of $A^*(\Hilb_n)$ are universal classes, see Definition \ref{def:universal}. During a conversation on Hilbert schemes, Alina Marian suggested that all elements of $A^*(\Hilb_n)$ should be universal, and the purpose of the present note is to prove it. \\

\begin{theorem}
\label{thm:main} 
	
Any element of $A^*(\eHilb_n)$ is a universal class. \\
	
\end{theorem}

\noindent When $S$ is projective, this result follows from an explicit formula for the diagonal of $\Hilb_n$ as a Chern class of the so-called $\text{Ext}$ virtual bundle, which in turn can be written in terms of universal classes (see \cite{ES} for $S = \BP^2$, \cite{B} for $S$ with effective anti-canonical line bundle, \cite{M} for $S$ with trivial canonical line bundle, and \cite{GT} for the general case). Our proof is quite different from those above, and holds for quasi-projective $S$. We start from \cite{dCM}, which states that:
\begin{equation}
\label{eqn:dcm}
A^*(\Hilb_n) \cong \mathop{\bigoplus^{k_1 + ... + k_t =  n}_{\Gamma \in A^*(S^t)^{\text{sym}}}}^{k_1 \geq ... \geq k_t} \mathfrak{C}_{k_1,...,k_t}(\Gamma)
\end{equation}
where $\mathfrak{C}_{k_1,...,k_t}(\Gamma)$ are certain correspondences (see \eqref{eqn:decomp} for an explicit description, as well as an explanation of the superscript ``sym") expressed in terms of the Heisenberg operators $\fq_k$ of \cite{G} and \cite{Nak}. The explicit description of these operators in terms of l.c.i. morphisms from \cite{Hecke} allows us to show that they preserve the subrings of $A^*(\Hilb_n)$ consisting of universal classes, thus implying Theorem \ref{thm:main}. Moreover, this gives an algorithm for computing the universal classes corresponding to the various summands in \eqref{eqn:dcm}, as we will explain on the last two pages. \\

\noindent I would like to thank Alina Marian, Eyal Markman, Davesh Maulik, Georg Oberdieck, Junliang Shen, Richard Thomas and Qizheng Yin for many interesting discussions on Hilbert schemes. I gratefully acknowledge the NSF grants DMS--1760264 and DMS--1845034, as well as support from the Alfred P. Sloan Foundation. \\

\section{Hilbert schemes}

\subsection{} Let $S$ be a smooth algebraic surface over $\BC$. Let $\Hilb_n = S^{[n]}$ denote the Hilbert scheme which parametrizes length $n$ subschemes of $S$, i.e. exact sequences:
$$
0 \rightarrow I \rightarrow \CO_S \rightarrow Z \rightarrow 0
$$
($I$ will be an ideal sheaf) where $\text{length}(Z) = n$. There exists a universal subscheme:
$$
\CZ_n \subset \Hilb_n \times S
$$
whose restriction to any $\{ Z \} \times S$ is precisely $\text{Spec } Z$ as a subscheme of $S$. Then:
\begin{equation}
\label{eqn:universal ses}
0 \rightarrow \CI_n \rightarrow \CO_{\Hilb_n \times S} \rightarrow \CO_{\CZ_n} \rightarrow 0
\end{equation}
is a short exact sequence of coherent sheaves on $\Hilb_n \times S$, flat over $\Hilb_n$. Let:
\begin{equation}
\label{eqn:big hilb}
\Hilb = \bigsqcup_{n=0}^\infty \Hilb_n
\end{equation}
The Hilbert scheme $\Hilb_n$ is well-known to be a smooth $2n$ dimensional variety, so we may consider its Chow rings $A^*(\Hilb_n)$, always with rational coefficients. Set:
$$
A^*(\Hilb) = \bigoplus_{n=0}^\infty A^*(\Hilb_n) 
$$
For any $k \in \BN$, we let $\pi : \Hilb_n \times S^k \rightarrow \Hilb_n$ denote the standard projection, and let $\CZ_n^{(i)} \subset \Hilb_n \times S^k$ denote the pull-back of $\CZ_n \subset \Hilb_n \times S$ via the $i$--th projection. \\

\begin{definition}
\label{def:universal}
	
A \textbf{universal class} is any element of $A^*(\eHilb_n)$ of the form:
\begin{equation}
\label{eqn:universal}
\pi_{*} \Big[ P(...,\ech_{j}(\CO_{\CZ^{(i)}_n}),...)^{1 \leq i \leq k}_{j \in \BN} \Big]
\end{equation}
$\forall k \in \BN$ and $\forall$ polynomials $P$ with coefficients pulled back from $A^*(S^k)$, such that: 
\begin{equation}
\label{eqn:support}
\emph{supp }P \subset \CZ_n^{(1)} \cap ... \cap \CZ_n^{(k)}
\end{equation}
(which implies that the push-forward in \eqref{eqn:universal} is well-defined).\footnote{The notion above is more general than either the small or big tautological classes considered in \cite{MN}; the reason is the lack in general of a Kunneth decomposition: $A^*(S^k) \not \cong A^*(S)^{\otimes k}$} \\
	
\end{definition}

\begin{proposition} 
\label{prop:ring}

The set: 
$$
A^*_{\etaut}(\eHilb_n) \subseteq A^*(\eHilb_n)
$$ 
of all universal classes is a subring. \\

\end{proposition} 
	
\begin{proof} It is clear that any $\BQ$--linear combination of universal classes is universal, even if they are defined with respect to different $k$'s. This is because any class of the form \eqref{eqn:universal} for a given $k$ is also of the form \eqref{eqn:universal} for $k+1$. Indeed, we have:
$$
\pi_{*} \Big[ P(...,\ch_{j}(\CO_{\CZ^{(i)}_n}),...)^{1 \leq i \leq k}_{j \in \BN} \Big] = ( \pi \circ \sigma)_* \Big[ \sigma^* \left( P(...,\ch_{j}(\CO_{\CZ^{(i)}_n}),...)^{1 \leq i \leq k}_{j \in \BN} \right) \cdot \Delta_{k,k+1} \Big]
$$
where $\Delta_{k,k+1}$ is the pull-back of the codimension 2 diagonal in $S^{k+1}$ involving the last two factors, and $\sigma : \Hilb_n \times S^{k+1} \rightarrow \Hilb_n \times S^k$ forgets the last copy of $S$. So it remains to prove that the product of universal classes is also a universal class. This is a consequence of the identity:
$$
\pi_{*} \Big[ P(...,\ch_{j}(\CO_{\CZ^{(i)}_n}),...)^{1 \leq i \leq k}_{j \in \BN} \Big] \cdot \rho_{*} \Big[ Q(...,\ch_{j}(\CO_{\CZ^{(i)}_n}),...)^{1  \leq i \leq l}_{j \in \BN} \Big] = 
$$
\begin{equation}
\label{eqn:identity}
= f_* \Big[ P(...,\ch_{j}(\CO_{\CZ^{(i)}_n}),...)^{1 \leq i \leq k}_{j \in \BN} \cdot Q(...,\ch_{j}(\CO_{\CZ^{(i)}_n}),...)^{k+1  \leq i \leq k+l}_{j \in \BN} \Big]
\end{equation}
with all maps as in the following Cartesian square:
$$
\xymatrix{
\Hilb_n \times S^{k+l} \ar[d]_{\rho'} \ar[r]^-{\pi'} \ar[rd]^f & \Hilb_n \times S^{l} \ar[d]^{\rho} \\
\Hilb_n \times S^{k}  \ar[r]^-{\pi}  & \Hilb_n }
$$
(all the maps are identities on $\Hilb_n$, and we think of $\pi'$ and $\rho'$ as forgetting the first, respectively last, factors of $S$). The identity \eqref{eqn:identity} is a straightforward consequence of $f = \rho \circ \pi' = \pi \circ \rho'$ and the base change formula $\pi'_* {\rho'}^* = \rho^* \pi_*$. Throughout the present proof, we were able to use the push-forward maps $\pi_*$, $\sigma_*$ and $\rho_*$ even if $\pi, \sigma$ and $\rho$ were non-proper (in the case of quasi-projective $S$), because we only applied them to classes whose support is proper under the respective maps. 
	
\end{proof} 

\subsection{}
\label{sub:nakajima}

We will prove Theorem \ref{thm:main} by deducing it from another well-known description of $A^*(\Hilb_n)$: the de Cataldo-Migliorini decomposition (\cite{dCM}). To review this construction, we must recall the Heisenberg algebra action introduced independently by Grojnowski (\cite{G}) and Nakajima (\cite{Nak}) on the Chow groups of Hilbert schemes. For any $n,k \in \BN$, consider the closed subscheme:
$$
\Hilb_{n,n+k} = \Big\{(I \supset I') \text{ s.t. } I/I' \text{ is supported at a single }x \in S \Big\} \subset \Hilb_n \times \Hilb_{n+k}
$$
endowed with projection maps:
\begin{equation}
\label{eqn:diagram zk}
\xymatrix{& \Hilb_{n,n+k} \ar[ld]_{p_-} \ar[d]^{p_S} \ar[rd]^{p_+} & \\ \Hilb_{n} & S & \Hilb_{n+k}}
\end{equation}
that remember $I$, $x$, $I'$, respectively. One may use $\Hilb_{n,n+k}$ as a correspondence:
\begin{equation}
\label{eqn:nakajima}
A^*(\Hilb) \xrightarrow{\fq_k} A^*(\Hilb \times S)
\end{equation}
(recall the notation \eqref{eqn:big hilb}) given by:
\begin{equation}
\label{eqn:nakajima formula}
\fq_k = (p_+ \times p_S)_* \circ p_-^*
\end{equation}
\footnote{The transposed correspondences give rise to operators $\fq_{-k}$, which we will not study} The main result of \cite{Nak} is that the operators $\fq_k$ obey the commutation relations in the Heisenberg algebra. More generally, we may consider:
\begin{equation}
\label{eqn:composition 1}
\fq_{k_1}...\fq_{k_t} : A^*(\Hilb) \rightarrow A^*(\Hilb \times S^t)
\end{equation}
where the convention is that the operator $\fq_{k_i}$ acts in the $i$-th factor of $S^t = S \times ... \times S$. Then associated to any $\Gamma \in A^*(S^t)$, one obtains an endomorphism of $A^*(\Hilb)$:
\begin{equation}
\label{eqn:composition 2}
\fq_{k_1}...\fq_{k_t}(\Gamma) = \pi_{1*} (\pi_2^*(\Gamma) \cdot \fq_{k_1}...\fq_{k_t})
\end{equation}
where $\pi_1, \pi_2 : \Hilb \times S^t \rightarrow \Hilb, S^t$ denote the standard projections (the non-properness of $\pi_1$ is not a problem for defining \eqref{eqn:composition 2}, because the support of $\fq_{k_1}...\fq_{k_t}$ is proper over $\Hilb$). One of the main results of \cite{dCM} is the following decomposition:
\begin{equation}
\label{eqn:decomp}
A^*(\Hilb) = \bigoplus^{k_1 \geq ... \geq k_t \in \BN}_{\Gamma \in A^*(S^t)^{\text{sym}}} \fq_{k_1}... \fq_{k_t}(\Gamma) \cdot A^*(\Hilb_0)
\end{equation}
where the superscript $``\text{sym}"$ refers the part of $A^*(S^t)$ which is symmetric with respect to those transpositions $(ij) \in \mathfrak{S}_t$ for which $k_i = k_j$. Since $\Hilb_0 = \text{pt}$, we have $A^*(\Hilb_0) \cong \BQ$, and so Theorem \ref{thm:main} follows from \eqref{eqn:decomp} and the following: \\

\begin{proposition}
\label{prop:preserve}

The endomorphisms \eqref{eqn:composition 2} preserve the subrings $A^*_{\etaut}(\eHilb_n)$. \\

\end{proposition}

\subsection{}
\label{sub:nested}

The remainder of our paper will be devoted to proving Proposition \ref{prop:preserve}. The problem with doing so directly from the definition \eqref{eqn:nakajima formula} is that the correspondences \eqref{eqn:diagram zk} are rather singular. The exception to this is the case $k=1$, namely:
\begin{equation}
\label{eqn:diagram z1}
\xymatrix{& \Hilb_{n-1,n} \ar[ld]_{p_-} \ar[d]^{p_S} \ar[rd]^{p_+} & \\ \Hilb_{n-1} & S & \Hilb_{n}} \qquad \xymatrix{& (I \supset_x I')  \ar[ld]_{p_-} \ar[d]^{p_S} \ar[rd]^{p_+} & \\ I & x & I'}
\end{equation}
Above and hereafter, we will write $I \supset_x I'$ if $I \supset I'$ and $I/I' \cong \BC_x$. It is well-known that $\Hilb_{n-1,n}$ is smooth of dimension $2n$. Consider the line bundle:
\begin{equation}
\label{eqn:taut 1}
\xymatrix{\CL \ar@{..>}[d] \\ \Hilb_{n-1,n}} \qquad \CL|_{(I \supset_x I')} = \Gamma(S,I/I')
\end{equation}
If $\CE = [\CW \rightarrow \CV]$ is a complex of locally free sheaves on a scheme $X$, then we define:
\begin{equation}
\label{eqn:subscheme}
\BP_X(\CE) \hookrightarrow \BP_X(\CV) := \text{Proj}_X(\text{Sym}(\CV))
\end{equation}
to be the closed subscheme determined by the image of the map: 
\begin{equation}
\label{eqn:equation}
\rho^*(\CW) \rightarrow \rho^*(\CV) \rightarrow \CO(1)
\end{equation}
where $\rho : \BP_X(\CV) \rightarrow X$ is the standard projection. In all cases considered in the present paper, the closed subscheme \eqref{eqn:subscheme} is a local complete intersection, cut out by the cosection \eqref{eqn:equation}. The following result is closely related to Lemma 1.1 of \cite{EGL}: \\

\begin{proposition}
\label{prop:1 minus}

Let $\CI_n$ be the universal ideal sheaf on $\eHilb_n \times S$, i.e. the kernel of the map $\CO_{\eHilb_n \times S} \twoheadrightarrow \CO_{\CZ_n}$ from \eqref{eqn:universal ses}. Then we have an isomorphism:
\begin{equation}
\label{eqn:p plus} 
\xymatrix{\eHilb_{n-1,n} \ar[r]^-{\cong} \ar[rd]_{p_+ \times p_S} & \BP_{\eHilb_n \times S}(- \CI_n^\vee \otimes \omega_S) \ar[d]^\rho \\ & \eHilb_n \times S}
\end{equation}
The line bundle $\CL$ on $\eHilb_{n-1,n}$ is isomorphic to $\CO(-1)$ on $\BP_{\eHilb_n \times S}(- \CI_n^\vee \otimes \omega_S)$. \\

\end{proposition}

\noindent We refer the reader to Section 4 of \cite{W surf} for details on why \eqref{eqn:p plus} is a special case of \eqref{eqn:subscheme}. In a few words, there is a short exact sequence with $\CW$, $\CV$ locally free:
\begin{equation}
\label{eqn:ses}
0 \rightarrow \CW \rightarrow \CV \rightarrow \CI_n \rightarrow 0
\end{equation}
Then the notation $- \CI^\vee_n$ in \eqref{eqn:p plus} stands for the complex $\left[ \CV^\vee \rightarrow \CW^\vee \right]$. Finally, $\omega_S$ denotes both the canonical line bundle on $S$ and its pull-back to $\Hilb_n \times S$. \\

\subsection{} 

Let us consider the following more complicated cousin of the scheme $\Hilb_{n,n+1}$:
\begin{equation}
\label{eqn:punctual}
\Hilb_{n-1,n,n+1} = \Big \{ (I, I',I'') \text{ such that } I \supset_x I' \supset_x I''
\text{for some } x \in S \Big \} 
\end{equation}
where $I \in \Hilb_{n-1}$, $I' \in \Hilb_{n}$ and $I'' \in \Hilb_{n+1}$. We have shown in \cite{W surf} that $\Hilb_{n-1,n,n+1}$ is smooth of dimension $2n+1$. Consider the line bundles:
\begin{equation}
\label{eqn:taut 2}
\xymatrix{\CL, \CL' \ar@{..>}[d] \\ \Hilb_{n-1,n,n+1}} \quad \CL|_{(I \supset_x I' \supset_x I'')} = \Gamma(S,I'/I''), \ \CL'|_{(I \supset_x I' \supset_x I'')} = \Gamma(S,I/I')
\end{equation}
Consider also the proper maps which forget either $I''$ or $I$:
\begin{equation}
\label{eqn:diagram z2}
\xymatrix{& \Hilb_{n-1,n,n+1} \ar[ld]_{\pi_-} \ar[d]^{\pi_+} \\ \Hilb_{n-1,n} & \Hilb_{n,n+1}}
 \qquad \xymatrix{& (I \supset_x I' \supset_x I'') \ar[ld]_{\pi_-} \ar[d]^{\pi_+} \\ (I \supset_x I') & (I' \supset_x I'')}
\end{equation}
Let $\Gamma : \Hilb_{n,n+1} \hookrightarrow \Hilb_{n,n+1} \times S$ be the graph of the map $p_S$, and let $\CL$ be the line bundle \eqref{eqn:taut 1} on $\Hilb_{n,n+1}$. We showed in \cite{Hecke} that there is a short exact sequence:
$$
0 \rightarrow \CL^{-1} \rightarrow \Gamma^*(\CV^\vee) \rightarrow \Gamma^*(\CW^\vee)
$$
with $\CW$, $\CV$ as in \eqref{eqn:ses}. We will use the notation $- \Gamma^*(\CI_n^\vee) + \CL^{-1}$ for the complex:
$$
\left[ \frac {\Gamma^*(\CV^\vee)}{\CL^{-1}} \longrightarrow \Gamma^*(\CW^\vee) \right]
$$
of coherent sheaves on $\Hilb_{n,n+1}$. \\

\begin{proposition}
\label{prop:2 minus}

(\cite{Hecke}) Let $\CI_n$ be the universal ideal sheaf on $\eHilb_n \times S$. Then:
\begin{equation}
\label{eqn:pi plus} 
\xymatrix{\eHilb_{n-1,n,n+1} \ar[r]^-{\cong} \ar[rd]_{\pi_+} & \BP_{\eHilb_{n,n+1}}(- \Gamma^*(\CI_n^\vee) \otimes \omega_S + \CL^{-1} \otimes \omega_S) \ar[d]^\rho \\ & \eHilb_{n,n+1}}
\end{equation} 
The line bundle $\CL'$ on $\eHilb_{n-1,n,n+1}$ is isomorphic to $\CO(-1)$ on the projectivization. \\

\end{proposition}

\noindent In both \eqref{eqn:p plus} and \eqref{eqn:pi plus}, we considered projectivization $\BP(*)$ where $*$ is written as a $K$--theory class instead of as a complex of sheaves. The reason for this is that we are only interested in $*$ inasmuch as it helps us compute push-forwards. In fact, the definition of Chern/Segre classes implies that we have, for all $k \geq 0$:
\begin{equation}
\label{eqn:push p}
(p_+ \times p_S)_*(c_1(\CL)^k) = (-1)^k c_{k+2} \left(\CI_n \otimes \omega_S^{-1}\right)
\end{equation}
\begin{equation}
\label{eqn:push pi}
\pi_{+*}(c_1(\CL')^k) = (-1)^k c_{k+1} \left(\Gamma^*(\CI_n) \otimes \omega_S^{-1} - \CL \otimes \omega_S^{-1} \right)
\end{equation}

\subsection{} Our reason for defining the smooth schemes $\Hilb_{n-1,n,n+1}$ is that it allows us to produce a resolution of the singular scheme $\Hilb_{n,n'}$, for any $n<n'$, in the following sense. Consider the following diagram of spaces and maps (\cite{Hecke}):
$$
\xymatrix{& \Hilb_{n,n+1,n+2} \ar[ld]_{\pi_-} \ar[rd]^{\pi_+} & & \Hilb_{n'-2,n'-1,n'} \ar[ld]_{\pi_-} \ar[rd]^{\pi_+} & & \\
\Hilb_{n,n+1} \ar[d]_{p_-} & & \dots & & \Hilb_{n'-1,n'} \ar[d]^{p_+ \times p_S} \\
\Hilb_n & & & & \Hilb_{n'} \times S}
$$
for all $n<n'$. Then we have the following formula (\cite{MN}):
\begin{equation}
\label{eqn:resolution}
\fq_k = (p_+ \times p_S)_* \circ (\pi_{+*} \circ \pi_-^*)^{k-1} \circ p_-^*
\end{equation}
Indeed, the right-hand side of \eqref{eqn:resolution} is a $2n+k+1$ dimensional cycle $\mathfrak{C}$ supported on the $2n+k+1$ dimensional locus $\Hilb_{n,n+k}$. It is well-known that the latter locus has a single irreducible component of top dimension, namely the closure of the locus $U$ of pairs $(I \supset I')$ where $I/I'$ is isomorphic to a length $k$ subscheme of a curve supported at a single point. But in this case, there exists a unique full flag of ideals $I = I_0 \supset I_1 \supset ... \supset I_k = I'$, which implies that $\mathfrak{C}|_U \cong U$, hence \eqref{eqn:resolution} follows. \\

\subsection{} For any $t \geq 0$, define the universal subring:
\begin{equation}
\label{eqn:taut subring 1}
A^*_\taut(\Hilb_n \times S^t) \subset A^*(\Hilb_n \times S^t)
\end{equation}  
as the subring generated by the classes \eqref{eqn:universal} for all $k \geq t$, where one replaces $\pi$ by the map $\Hilb_n \times S^{k} \rightarrow \Hilb_n \times S^t$ which forgets the last $k-t$ factors. Then we let:
\begin{align} 
&A^*_\taut(\Hilb_{n,n+1} \times S^t) \subset A^*(\Hilb_{n,n+1} \times S^t) \label{eqn:taut subring 2} \\
&A^*_\taut(\Hilb_{n-1,n,n+1} \times S^t) \subset A^*(\Hilb_{n-1,n,n+1} \times S^t) \label{eqn:taut subring 3}
\end{align} 
be the subrings generated by $c_1(\CL)$ (respectively $c_1(\CL')$) and the pull-backs of all universal classes via the following maps, respectively: 
\begin{align*} 
&p_- \times \Id_{S^t} : \Hilb_{n,n+1} \times S^t \rightarrow \Hilb_n \times S^t \\
&\pi_- \times \Id_{S^t} : \Hilb_{n-1,n,n+1} \times S^t \rightarrow \Hilb_{n-1,n} \times S^t
\end{align*} 
With this in mind, Proposition \ref{prop:preserve} is a consequence of \eqref{eqn:resolution} and the following: \\

\begin{proposition}
\label{prop:final}
	
For any $t \geq 0$, the maps $(p_- \times \emph{Id}_{S^t})^*$, $(\pi_- \times \emph{Id}_{S^t})^*$, $(\pi_{+} \times \emph{Id}_{S^t})_*$, $(p_+ \times p_S \times \emph{Id}_{S^t})_*$ preserve the universal subrings, as defined above. \\

\end{proposition} 

\noindent Indeed, formula \eqref{eqn:resolution} and Proposition \ref{prop:final} imply that if $x \in A^*_{\taut}(\Hilb)$, then $y = \fq_{k_1}...\fq_{k_t}(x) \in A^*_{\taut}(\Hilb \times S^t)$, in the sense of \eqref{eqn:taut subring 1}. If we multiply $y$ by the pull-back of any $\Gamma \in A^*(S^t)$ and then push it forward to $\Hilb$, it will remain in the subring of universal classes, and this establishes Proposition \ref{prop:preserve}. \\

\begin{proof} \emph{of Proposition \ref{prop:final}:} The statements about the pull-back maps $(p_- \times \Id_{S^t})^*$ and $(\pi_- \times \Id_{S^t})^*$ preserving the universal rings are obvious given definitions \eqref{eqn:taut subring 2} and \eqref{eqn:taut subring 3}. Concerning the push-forward $(\pi_{+} \times \Id_{S^t})_*$, we must show that:
\begin{equation}
\label{eqn:show 1}
x \in A^*_\taut(\Hilb_{n-1,n,n+1} \times S^t) \Rightarrow (\pi_{+} \times \Id_{S^t})_*(x) \in A^*_\taut(\Hilb_{n,n+1} \times S^t)
\end{equation}
We have the following short exact sequence on $\Hilb_{n-1,n,n+1} \times S$:
\begin{equation}
\label{eqn:short}
0 \rightarrow \CL' \otimes (p_S \times \Id_S)^*(\CO_\Delta) \rightarrow \CO_{\mathcal{Z}_n} \rightarrow \CO_{\mathcal{Z}_{n-1}} \rightarrow 0 
\end{equation}
where $\Delta \subset S \times S$ is the diagonal, and $p_S : \Hilb_{n-1,n,n+1} \rightarrow S$ is the map which remembers the point $x$ in \eqref{eqn:punctual}. Then for any polynomial $P$ in the Chern classes of $\CZ_{n-1}^{(i)}$, whose coefficients are pulled back from $\Hilb_{n,n+1} \times S^k$ to $\Hilb_{n-1,n,n+1} \times S^k$:
$$
P(...,\ch_{j}(\CO_{\CZ^{(i)}_{n-1}}),...)^{1 \leq i \leq k}_{j \in \BN} = P(...,\ch_{j}(\CO_{\CZ^{(i)}_{n}} - \CL' \otimes (p_S \times \text{proj}_i)^*(\CO_\Delta)),...)^{1 \leq i \leq k}_{j \in \BN} = 
$$
\begin{equation}
\label{eqn:expression 1}
= \sum_{a=0}^\infty c_1(\CL')^a \cdot (\pi_{+} \times \Id_{S^k})^*(R_a) \ \in \  A^*(\Hilb_{n-1,n,n+1} \times S^k)
\end{equation}
for various $R_a \in A^*(\Hilb_{n,n+1} \times S^k)$ which are also polynomials in the Chern classes of the universal subschemes. If we apply $(\pi_+ \times \Id_{S^k})_*$ to \eqref{eqn:expression 1}, we obtain:
\begin{equation}
\label{eqn:temp 1}
(\pi_+ \times \Id_{S^k})_* \left[ P(...,\ch_{j}(\CO_{\CZ^{(i)}_{n-1}}),...)^{1 \leq i \leq k}_{j \in \BN} \right] = \sum_{a=0}^\infty \pi_{+*} (c_1(\CL')^a) \cdot R_a
\end{equation}
Letting $\int_{S^{k-t}}$ denote the push-forward along the last $k-t$ factors of $S^k$, we have:
$$
(\pi_+ \times \Id_{S^t})_* \underbrace{\int_{S^{k-t}} P(...,\ch_{j}(\CO_{\CZ^{(i)}_{n-1}}),...)^{1 \leq i \leq k}_{j \in \BN}}_{x \text{ from the LHS of \eqref{eqn:show 1} is of this form}} = \sum_{a=0}^\infty \pi_{+*} (c_1(\CL')^a) \int_{S^{k-t}} R_a
$$
The implication \eqref{eqn:show 1} follows because the RHS above is a universal class: this is true for $\pi_{+*}(c_1(\CL')^a)$ because of formula \eqref{eqn:push pi}, and for $\int_{S^{k-t}} R_a$ since $R_a$ is a polynomial in the Chern classes of the universal subscheme. Finally, the support condition \eqref{eqn:support} is satisfied becase the LHS of \eqref{eqn:expression 1} is by assumption supported on the locus $(I_{n-1} \supset_x I_n \supset_x I_{n+1}) \times (x_1,...,x_k) \in \Hilb_{n-1,n,n+1} \times S^k$, where $x_1,...,x_k \in \text{supp }I_{n-1}$, hence the $R_a$'s in the RHS of \eqref{eqn:temp 1} are supported on the locus $(I_n \supset_x I_{n+1}) \times (x_1,...,x_k) \in \Hilb_{n,n+1} \times S^k$, where $x_1,...,x_k \in \text{supp }I_{n}$. \\

\noindent Concerning the push-forward $(p_+ \times p_S \times \Id_{S^t})_*$, we must show that:
\begin{equation}
\label{eqn:show 2}
y \in A^*_\taut(\Hilb_{n-1,n} \times S^t) \Rightarrow (p_+ \times p_S \times \Id_{S^t})_*(y) \in A^*_\taut(\Hilb_{n} \times S \times S^{t})
\end{equation}
By a short exact sequence analogous to \eqref{eqn:short}, for any polynomial $P$ in the Chern classes of $\CZ_{n-1}^{(i)}$, whose coefficients are pulled back from $\Hilb_{n} \times S^k$ to $\Hilb_{n-1,n} \times S^k$:
\begin{multline}
P(...,\ch_{j}(\CO_{\CZ^{(i)}_{n-1}}),...)^{1 \leq i \leq k}_{j \in \BN} = \\ = \sum_{a=0}^\infty c_1(\CL)^a \cdot (p_+ \times p_S \times \Id_{S^k})^*(R_a) \ \in \  A^*(\Hilb_{n-1,n} \times S^k) \label{eqn:expression 2}
\end{multline}
for various $R_a \in A^*(\Hilb_{n} \times S \times S^{k})$ which are polynomials in the Chern classes of the universal subschemes. If we apply $(p_+ \times p_S \times \Id_{S^k})_*$ to \eqref{eqn:expression 2}, we obtain:
\begin{equation}
\label{eqn:temp 2}
(p_+ \times p_S \times \Id_{S^k})_* P(...,\ch_{j}(\CO_{\CZ^{(i)}_{n-1}}),...)^{1 \leq i \leq k}_{j \in \BN} = \sum_{a=0}^\infty (p_+ \times p_S)_* (c_1(\CL)^a) \cdot R_a
\end{equation}
hence:
$$
(p_+ \times p_S \times \Id_{S^t})_* \underbrace{\int_{S^{k-t}} P(...,\ch_{j}(\CO_{\CZ^{(i)}_{n-1}}),...)^{1 \leq i \leq k}_{j \in \BN}}_{y \text{ from the LHS of \eqref{eqn:show 2} is of this form}} = \sum_{a=0}^\infty (p_+ \times p_S)_* (c_1(\CL)^a) \int_{S^{k-t}} R_a
$$
The implication \eqref{eqn:show 2} follows because the RHS above is a universal class: this is true for $(p_+ \times p_S)_*(c_1(\CL)^a)$ because of formula \eqref{eqn:push p}, and for $\int_{S^{k-t}} R_a$ since $R_a$ is a polynomial in the Chern classes of the universal subscheme. However, checking the support condition is non-trivial, so let us explain this. By assumption, the LHS of \eqref{eqn:expression 2} is supported on the locus of points $(I_{n-1} \supset_x I_n) \times (x_1,...,x_k) \in \Hilb_{n-1,n} \times S^k$, where $x_1,...,x_k \in \text{supp } I_{n-1}$. Therefore, the $R_a$'s which appear in the RHS of \eqref{eqn:temp 2} are supported on the locus of points $(I_n,x) \times (x_1,...,x_k) \in \Hilb_{n} \times S \times S^k$, where $x_1,...,x_k \in \text{supp } I_{n}$. However, a universal class also needs to be supported on the locus of points where $x \in \text{supp }I_n$, but we are rescued because:
$$
(p_+ \times p_S)_* (c_1(\CL)^a) = (-1)^a c_{a+2} \left(\CI_n \otimes \omega_S^{-1}\right)
$$
vanishes on the locus $x \notin \text{supp }I_n$ (this holds because the universal ideal sheaf $\CI$ is trivial on the locus $x \notin \text{supp }I_n$, and $c_{a+2}(\text{line bundle}) = 0$ for $a \geq 0$). 
\end{proof}

\end{document}